\documentclass[12pt,reqno]{amsart}

\usepackage{multirow}
\usepackage{hhline}

\usepackage{times}
\usepackage[T1]{fontenc}
\usepackage{mathrsfs}
\usepackage{latexsym}
\usepackage[dvips]{graphics}
\usepackage[titletoc,title]{appendix}
\usepackage{epsfig}

\usepackage{amsmath,amsfonts,amsthm,amssymb,amscd}
\input amssym.def
\input amssym.tex
\usepackage{color}
\usepackage{hyperref}

\usepackage{color}
\usepackage{breakurl}

\usepackage{comment}
\newcommand{\bburl}[1]{\textcolor{blue}{\url{#1}}}

\newcommand{\be}{\begin{equation}}
\newcommand{\ee}{\end{equation}}
\newcommand{\bea}{\begin{eqnarray}}
\newcommand{\eea}{\end{eqnarray}}

\newtheorem{thm}{Theorem}[section]

\newtheorem{cor}[thm]{Corollary}
\newtheorem{lem}[thm]{Lemma}

\newtheorem{defi}[thm]{Definition}
\newtheorem{rek}[thm]{Remark}

\numberwithin{equation}{section}

\begin{document}

\title{Difference in the Number of Summands in the Zeckendorf Partitions of Consecutive Integers}

\author{H\`ung Vi\d{\^e}t Chu}
\email{\textcolor{blue}{\href{mailto:hungchu2@illinois.edu}{hungchu2@illinois.edu}}}
\address{Department of Mathematics, University of Illinois at Urbana-Champaign, Urbana, IL 61820, USA}



\thanks{}

\date{\today}
\maketitle

\begin{abstract}
Zeckendorf proved that every positive integer has a unique partition as a sum of non-consecutive Fibonacci numbers. We study the difference between the number of summands in the partition of two consecutive integers. In particular, let $L(n)$ be the number of summands in the partition of $n$. We characterize all positive integers such that $L(n) > L(n+1)$, $L(n) < L(n+1)$, and $L(n) = L(n+1)$. Furthermore, we call $n+1$ a \textit{peak} of $L$ if $L(n) < L(n+1) > L(n+2)$ and a \textit{divot} of $L$ if $L(n) > L(n+1) < L(n+2)$. We characterize all such peaks and divots of $L$. 
\end{abstract}

\section{Introduction}
The Fibonacci sequence $\left\{F_n\right\}_{n=0}^{\infty}$ is defined as follows: let $F_0 = 0$, $F_1 = 1$, and $F_n = F_{n-1} + F_{n-2}$, for $n\ge 2$.
A beautiful theorem of Zeckendorf \cite{Z} states that every positive integer $n$ can be uniquely written as a sum of non-consecutive Fibonacci numbers. This gives the so-called Zeckendorf partition of $n$. A formal statement of Zeckendorf's theorem is as follows:
\begin{thm}[Zeckendorf]\label{p1}
For any $n\in\mathbb{N}$, there exists a unique increasing sequence of positive integers $(c_1, c_2, \ldots, c_k)$ such that $c_1\ge 2$, $c_i\ge c_{i-1}+2$ for $i = 2, 3, \ldots, k$, and $n = \sum_{i=1}^kF_{c_i}$. 
\end{thm}
Much work has been done to understand the structure of Zeckendorf partitions and their applications (see \cite{BDEMMTW1, BDEMMTW2, B, CSH, Fr, MG1, MG2, HS, L, MW1, MW2}) and to generalize them (see \cite{Day, DDKMMV, DFFHMPP, FGNPT, GTNP, Ho, K, ML, MMMS, MMMMS}). Let $\phi = \frac{1+\sqrt{5}}{2}$, the golden ratio. Many properties of the number of summands in Zeckendorf partitions have been discovered. For example, Lekkerkerker \cite{L} proved that the average number of summands needed for $x\in [F_n,F_{n+1})$ as $n \rightarrow \infty$ is $n/(\phi^2+1)$. Furthermore, Kologlu et al., \cite{KKMW} investigated the distribution of the number of summands and showed that the distribution converges to a Gaussian as $n \rightarrow\infty$. Continuing these work, we study the difference in the number of summands of consecutive integers.

Let $P(n)$ be the set of Fibonacci numbers in the Zeckendorf partition of $n$ and $L(n)$ be the number of summands in the Zeckendorf partition of $n$; that is, $L(n) := \#P(n)$. Define the arithmetic function $f(n) := L(n+1) - L(n)$, which measures the difference between numbers of summands of two consecutive integers. Our main result characterizes all $n$ such that $f(n) > 0$, $f(n) = 0$, and $f(n) < 0$. The following data is generated by a {\tt Python} program.

\begin{center}
    \begin{tabular}{|c|c|c|c|}
    \hline
      $N$ & $f(n) > 0$ & $f(n) < 0$ & $f(n) = 0$ \\
      \hline
      10 & 3 & 2 & 5  \\ \hline
      100 & 38 & 23 & 39 \\ \hline
      1,000 & 382 & 236 & 382 \\ \hline 
      10,000 & 3819 & 2360 & 3820 \\ \hline
      $10^5$ & 38196 & 23606 & 38197\\ \hline
      $10^6$ & 381966 & 236068 & 381966\\ \hline
    \end{tabular}
\end{center}

\begin{center}
Table 1. Number of positive integers at most $N$ such that $f(n) > 0$, $f(n) < 0$, and $f(n) = 0$, respectively. 
\end{center}

Table 1 indicates that the number of positive integers $n$ with $f(n) > 0$ and $f(n) = 0$ are extremely close, while the number of $n$ with $f(n) < 0$ is smaller. The data suggest that as $N\rightarrow \infty$, the probability that a positive integer at most $N$ satisfying $f(n) < 0$ is equal to the probability that $f(n) = 0$, which is greater than the probability that $f(n) < 0$. As a consequence, the number of summands in the Zeckendorf partition of $n+1$ is more likely to be equal to or greater than the number of summands in the partition of $n$. We state our first result.

\begin{thm} \label{maintheo} For positive integers $n$, we have $L(n) < L(n+1)$ if and only if \begin{align}\label{f1}S_1\ : = \ \left\{\left\lfloor{\frac{n+1}{\phi}}\right\rfloor+2n\,:\, n\ge 1\right\}.\end{align}
We have $L(n) > L(n+1)$ if and only if \begin{align}\label{f2}S_2\ :=\ \left\{2\left\lfloor\frac{n+1}{\phi}\right\rfloor+3n-1\, :\, n\ge 1\right\}.\end{align}
\end{thm}

\begin{cor}\label{maincor}
For an arbitrary positive integer $n$, the probability that 
 \begin{enumerate}
\item  $f(n) < 0$ is $\frac{\phi}{3\phi+2} = 0.2360\ldots$,
\item $f(n) = 0$ is $\frac{\phi}{1+2\phi} = 0.38196\ldots$,
\item $f(n) > 0$ is $\frac{\phi}{1+ 2\phi} = 0.38196\ldots$.
 \end{enumerate}
\end{cor}

\begin{proof}
Let $N\in \mathbb{N}$. Due to \eqref{f1}, the number of positive integers that are at most $N$ and in $S_1$ is 
$$\frac{N}{2+1/\phi} + O(1).$$
These are also the positive integers $n\le N$ with $f(n)>0$. Hence, the probability that $f(n) > 0$ is 
$$\lim_{N\rightarrow \infty}\frac{1}{N}\left(\frac{N}{2+ 1/\phi}+O(1)\right) \ =\  \frac{\phi}{1+ 2\phi}.$$
Using \eqref{f2}, we can compute the probability that $f(n) < 0$ is 
$$\lim_{N\rightarrow \infty}\frac{1}{N}\left(\frac{N}{2/\phi+3}+O(1)\right) \ =\ \frac{\phi}{3\phi+2}.$$
Finally, the probability that $f(n) = 0$ is 
$$1 - \frac{\phi}{1+ 2\phi} - \frac{\phi}{3\phi+2} \ =\ \frac{\phi}{1+2\phi}.$$
This completes our proof.
\end{proof}

A natural question is how large can $|f(n)|$ be. The following theorem provides the answer. 

\begin{thm}\label{minortheo1} Let $n$ be a positive integer. If $f(n) > 0$, then $f(n) = 1$. If $f(n) < 0$, then $f(n)$ can be equal to any negative integers. 
\end{thm}

Before stating our final result. We introduce a definition. 
\begin{defi}\normalfont
Given an arithmetic function $L$, if $L(n) > L(n+1) < L(n+2)$, then we call $n+1$ a \textit{divot} of the function $L$. If $L(n) < L(n+1) > L(n+2)$, then we call $n+1$ a \textit{peak} of the function $L$. 
\end{defi}

\begin{thm}\label{maintheo3} The following are true about $L$. 
\begin{enumerate}
\item There is no positive integer $n$ such that either $L(n) < L(n+1) < L(n+2)$ or $L(n) > L(n+1) > L(n+2)$. 
\item The set of positive integers $n$ such that $n+1$ is a divot is $S_2$. 
\item The set of positive integers $n$ such that $n+1$ is a peak is $$S_3 := \left\{3\left\lfloor\frac{n}{\phi}+\phi\right\rfloor + 5n\, :\, n\ge 0\right\}.$$
\end{enumerate}
\end{thm}

\section{Proofs}
We restate \cite[Theorem 3.4]{MG1}, which is a key ingredient in our proof. 
\begin{thm}\label{MGtheorem}
For $k\ge 2$, the set of all positive integers having the summand $F_k$ in their Zeckendorf partition is given by
\begin{align}\label{k1}
    Z(k) \ =\ \left\{F_k\left\lfloor\frac{n+\phi^2}{\phi}\right\rfloor+nF_{k+1}+j\, :\, 0\le j\le F_{k-1}-1, n\ge 0\right\}.
\end{align}
\end{thm}

Another formula that we will need is taken from \cite[Page 335]{MG1}. For $k\ge 2$, the following formula gives the set of all integers with $F_k$ and $F_{k+2}$ in their Zeckendorf partitions:
\begin{align}\label{k2}
    Z(k, k+2) := \left\{F_{k+2}\left\lfloor \frac{n+\phi^2}{\phi}\right\rfloor + nF_{k+3} +F_k + j\, :\, 0\le j\le F_{k-1}-1, n\ge 0\right\}.
\end{align}

\begin{lem}\label{m1}For each $n\in \mathbb{N}$, $L(n) < L(n+1)$ if and only if $P(n+1)$ contains $F_2$.\end{lem}

\begin{proof}
The backward implication is straightforward. If $P(n+1)$ contains $F_2$, then the Zeckendorf partition of $n$ can be found by getting rid of $F_2$ in the partition of $n+1$. Hence, $P(n) < P(n+1)$. We prove the forward. Suppose, for a contradiction, that $F_2\in P(n)$. If $n = F_2$, then $1 + n = F_3$, which contradicts that $L(n) < L(n+1)$. If $n \neq F_2$, write $n = F_2 + F_{m_1} + F_{m_2} + \cdots + F_{m_k}$ as the Zeckendorf partition of $n$. We have
$1 + n = 1 + F_2 + F_{m_1} + F_{m_2} + \cdots + F_{m_k} = F_3 + F_{m_1} + F_{m_2} + \cdots + F_{m_k}$. If $m_1 - 3 \ge 2$, then we have the Zeckendorf partition of $n+1$ and $L(n) = L(n+1)$, a contradiction. If $m_1-3 = 1$, we replace $F_3 + F_{m_1}$ by $F_5$ and repeat process until we have a Zeckendorf partition; however, $L(n+1) < L(n)$, a contradiction. Therefore, $F_2\notin P(n)$. Similarly, we can show that $F_3\notin P(n)$ and so, the smallest summand $F_k$ in $P(n)$ must have $k\ge 4$. This implies $P(n+1)$ contains $F_2$, as desired. 
\end{proof}

\begin{cor}\label{c1} The set of all $n$ with $L(n) < L(n+1)$ is 
$$\{\lfloor{(n+\phi^2)/\phi}\rfloor+2n-1\,:\, n\ge 1\}.$$
\end{cor}
\begin{proof}
By Theorem \ref{MGtheorem}, the set of all natural numbers with $F_2$ in their Zeckendorf partition is 
$$\left\{\left\lfloor{\frac{n}{\phi}+\phi}\right\rfloor+2n\,:\, n\ge 0\right\},$$
which implies our corollary.  
\end{proof}

\begin{lem}\label{m2} For each $n\in \mathbb{N}$, $L(n) > L(n+1)$ if and only if $P(n+1)$ contains none of $F_2, F_3, F_4$. \end{lem}

\begin{proof} We prove the forward implication. We know that $F_2\notin P(n+1)$ due to Lemma \ref{m1}. Suppose, for a contradiction, that $F_3\in P(n+1)$. If $n+1 = F_3$, then $n = 1  = F_2$ and $L(n) = L(n+1)$, a contradiction. If $n+1 > F_3$, write $n+1 = F_3 + F_{m_1} + F_{m_2} + \cdots + F_{m_k}$ as the Zeckendorf partition of $n+1$. Then $n = F_2 + F_{m_1} + F_{m_2} + \cdots + F_{m_k}$ is the Zeckendorf partition of $n$. Again, $L(n) = L(n+1)$, a contradiction. Hence, $F_3\notin P(n+1)$. The same argument shows that $F_{4}\notin P(n+1)$. 

To prove the backward implication, write $n+1 = F_{m_1} + F_{m_2} + \cdots + F_{m_k}$ to be the Zeckendorf partition of $n+1$ and $m_1\ge 5$. Then $n = (F_{m_1}-1) + F_{m_2} + \cdots + F_{m_k}$. Because $F_{m_1}-1$ is not a Fibonacci number due to $m_1\ge 5$, the Zeckendorf partition of $F_{m_1}-1$ contains at least two summands. Therefore, $L(n) > L(n+1)$, as desired.  
\end{proof}

\begin{lem}\label{m3} For each $n\in \mathbb{N}$, $P(n+1)$ contains none of $F_2, F_3, F_4$ if and only if $P(n+3)$ contains $F_3$.
\end{lem}

\begin{proof}
Suppose that $P(n+1)\cap \{F_2, F_3, F_4\} = \emptyset$. Then $n+1 = F_{m_1} + F_{m_2} + \cdots + F_{m_k}$ is the Zeckendorf partition of $n+1$ with $m_1 \ge 5$. Hence, $n+3 = F_3 + F_{m_1} + F_{m_2} + \cdots + F_{m_k}$ is the Zeckendorf partition of $n+3$ and so, $F_3\in P(n+3)$. 

Next, suppose that $P(n+3)$ contains $F_3$; write $n+3 = F_3 + F_{m_1} + F_{m_2} + \cdots + F_{m_k}$ is the Zeckendorf partition of $n+3$. Note that $m_1\ge 5$. Then $n+1 = F_{m_1} + F_{m_2}  + \cdots + F_{m_k}$ is the Zeckendorf partition of $n+1$ and so, $P(n+1)$ contains none of $F_2, F_3, F_4$.
\end{proof}

\begin{rek}\label{r1}\normalfont
By \eqref{k1}, the set of such all natural numbers $n$ such that $P(n+3)$ contains $F_3$ is 
$$\left\{2\left\lfloor\frac{n+\phi^2}{\phi}\right\rfloor+3(n-1)\, :\, n\ge 1\right\}.$$
\end{rek}

Due to Lemmas \ref{m2} and \ref{m3} and Remark \ref{r1}, we have the following corollary. 
\begin{cor}\label{c2}
 For each $n\in \mathbb{N}$, $L(n) > L(n+1)$ if and only if $$\left\{2\left\lfloor\frac{n+\phi^2}{\phi}\right\rfloor+3(n-1)\, :\, n\ge 1\right\}.$$
\end{cor}

\begin{proof}[Proof of Theorem \ref{maintheo}]
The proof follows directly from Corollaries \ref{c1} and \ref{c2} and the identity $\phi - 1 = \frac{1}{\phi}$.
\end{proof}

\begin{proof}[Proof of Theorem \ref{minortheo1}]
Due to Lemma \ref{m1}, if $f(n) > 0$, then $P(n+1)$ contains $F_{2}$. Hence, we obtain the Zeckendorf partition of $n$ by discarding $F_2$ in the partition of $n+1$. Therefore, $f(n) = L(n+1) - L(n) = 1$. 

To prove that $f(n)$ can be equal to any negative integers, we use the identity $1 + F_2 + \cdots + F_{2k} = F_{2k+1}$, which holds for all $k\ge 1$. Let $n = F_{2k+1} - 1 = F_2 + F_4 + \cdots + F_{2k}$ and $n+1 = F_{2k+1}$. We have $f(n) = L(F_{2k+1}) - L(F_{2k+1}-1) = 1 - k$. This completes our proof. 
\end{proof}

The following lemma is used in the proof of Theorem \ref{maintheo3}. 

\begin{lem}\label{m10} For each positive integer $n$, the following are equivalent
\begin{enumerate}
    \item $F_2\in P(n+1)$ and $P(n+2)\cap \{F_2, F_3, F_4\} = \emptyset$,
    \item $\{F_2, F_4\}\subset P(n+1)$.
\end{enumerate}
\end{lem}

\begin{proof}
Assume $(1)$. It suffices to show that $F_4\in P(n+1)$. Suppose otherwise. Then the next term in $P(n+1)$ (if any) is $F_{k}$, where $k\ge 5$. Hence, $F_3\in P(n+2)$, a contradiction. 

Assume $(2)$. We can write $n+1 = F_2 + F_4 + F_{m_1} + \cdots + F_{m_k}$ as the Zeckendorf partition of $n+1$. Then $n+2 = F_5 + F_{m_1} + \cdots + F_{m_k}$, which can be simplified further to have the Zeckendorf partition of $n+2$, which contains none of $F_2, F_3$, and $F_4$. 
\end{proof}

\begin{proof}[Proof of Theorem \ref{maintheo3}]
Suppose, for a contradiction, that $L(n) < L(n+1) < L(n+2)$ for some positive integer $n$. By Lemma \ref{m1}, both  $P(n+1)$ and $P(n+2)$ contain $F_2$. Write $n+1 = F_2 + F_{m_1} + \cdots + F_{m_k}$ as the Zeckendorf partition of $n+1$. Then $n+2 = (1 + F_2) + F_{m_1} + \cdots + F_{m_k}$. We can replace $1 + F_2$ by $F_3$. If $m_1 > 4$, then $n+2 = F_3 + F_{m_1} + \cdots + F_{m_k}$ is the Zeckendorf partition of $n+2$, which does not contain $F_2$. If $m_1 = 4$, then we can replace $F_3 + F_4$ by $F_5$ and repeat the process until we have a Zeckendorf partition, which does not contain $F_2$. In both cases, $F_2\notin P(n+2)$, a contradiction. Therefore, there is no $n$ such that $L(n) < L(n+1) < L(n+2)$.

Next, suppose, for a contradiction, that $L(n) > L(n+1) > L(n+2)$ for some positive integer $n$. By Lemma \ref{m2}, we know that $P(n+1)$ and $P(n+2)$ contains none of $F_2, F_3$, and $F_4$. However, if $n+1$ contains none of $F_2, F_3$, and $F_4$, then $n+2$ contains $F_2$, a contradiction. Hence, there is no $n$ such that $L(n) > L(n+1) > L(n+2)$.

To prove that $\{n\,:\, L(n) > L(n+1) < L(n+2)\} = S_2$, it suffices to prove that $L(n) > L(n+1)$ implies $L(n+1) < L(n+2)$. Assume $L(n) > L(n+1)$.  By Lemma \ref{m2}, we know that $P(n+1)$ contains none of $F_2, F_3$, and $F_4$. Therefore, $F_2\in P(n+2)$. By Lemma \ref{m1}, it follows that $L(n+1) < L(n+2)$. 

Finally, due to Lemmas \ref{m1} and \ref{m2}, the condition $L(n) < L(n+1) > L(n+2)$ is equivalent to two conditions $F_2\in P(n+1)$, and $P(n+2)$ contains none of $F_2, F_3$, and $F_4$. By Lemma \ref{m10}, we have $\{F_2, F_4\}\subset P(n+1)$. By \eqref{k2}, we know that $n$ is in
\begin{align*}
    \left\{3\left\lfloor\frac{n}{\phi}+\phi\right\rfloor + 5n\, :\, n\ge 0\right\}.
\end{align*}
\end{proof}

%%%%%%%%%%%%%%%%%%%%%%%%%%%%%%%%%%%%%%%%%%%%%%%%%%%%%%%%%%%%%%%%%%%%%%%%%
%%%%%%%%%%%%%%%%%%%%%%%%%%%%%%%%%%%%%%%%%%%%%%%%%%%%%%%%%%%%%%%%%%%%%%%%%
%%%%%%%%%%%%%%%%%%%%%%%%%%%%%%%%%%%%%%%%%%%%%%%%%%%%%%%%%%%%%%%%%%%%%%%%%
%%%%%%%%%%%%%%%%%%%%%%%%%%%%%%%%%%%%%%%%%%%%%%%%%%%%%%%%%%%%%%%%%%%%%%%%%
%%%%%%%%%%%%%%%%%%%%%%%%%%%%%%%%%%%%%%%%%%%%%%%%%%%%%%%%%%%%%%%%%%%%%%%%%

\ \\
\end{document}